\documentclass[3p,sort & compress]{elsarticle}
\usepackage{setspace}
\usepackage{amssymb}
\usepackage{latexsym}
\usepackage{amsbsy,amsfonts,amsthm,amsmath,amssymb,bbm,calc,caption,color,dsfont,graphics}
\usepackage{graphicx,ifthen,latexsym,mathrsfs,multicol}
\usepackage{epstopdf}
\usepackage[colorlinks,linkcolor=red]{hyperref}
\usepackage{multirow}
\usepackage{epsfig}
\usepackage{subfigure}
\usepackage{algorithm}
\usepackage{verbatim}
\usepackage{algorithmic}
\usepackage{natbib,overpic,pifont,psfrag,rotating,stmaryrd}
\usepackage{bm}
\usepackage{setspace}
\newtheorem{theorem}{Theorem}
\newtheorem{lemma}[theorem]{Lemma}

\newdefinition{example}{Example}

\numberwithin{equation}{section}

\numberwithin{theorem}{section}
\journal{}
\begin{document}
\begin{frontmatter}
\title{A Weak Galerkin Method with Implicit $\theta$-schemes for Second-Order Parabolic Problems}
\author{Wenya Qi\corref{qwy}}\ead{qiwy16@lzu.edu.cn}
\cortext[qwy]{Corresponding author.}
\address{School of Mathematics and Statistics, Lanzhou University, Lanzhou 730000, PR China}
\begin{abstract}
We introduce a new weak Galerkin finite element method whose weak functions on interior neighboring edges are double-valued for parabolic problems. Based on $(P_k(T), P_{k}(e), RT_k(T))$ element, a fully discrete approach is formulated with implicit $\theta$-schemes in time for $\frac{1}{2}\leq\theta\leq 1$, which include first-order backward Euler and second-order Crank-Nicolson schemes. Moreover, the optimal convergence rates in the $L^2$ and energy norms are derived. Numerical example is given to verify the theory.
  \\
\end{abstract}
\end{frontmatter}

\section{Introduction}\label{sec:1}
In this paper, an extension of weak Galerkin finite element method (WG) in \cite{Wang2013A} to parabolic problems will be introduced, and referred to over-penalized weak Galerkin finite element method (OPWG). Different from single-valued weak functions on interior edges in WG (\cite{Wang2013A}, \cite{Mu2015poly}, \cite{Mu2013Weakinterface1}), double-valued weak functions appeared in \cite{liu2018an} has been employed to strengthen flexibility of WG with $(P_k(T), P_{k}(e), RT_k(T))$ element. For realizing weak continuity of WG, we naturally deal with jumps on the interior edges by the penalty terms. Importantly,  penalized terms on weak functions will be analyzed with sharp penalized parameters explicitly given.

Let $\Omega\in \mathbb{R}^d ~(d=2,3)$ be an open and bounded polygonal or polyhedral domain. A linear parabolic model is listed as follows
\begin{align}
u_t-\nabla \cdot (A \nabla u)&=f(x,t),~~~~~~~~~~\mbox{in}~\Omega\times(0,\bar{T}], \label{eq:problem1.1}\\
u&=g(x,t),~~~~~~~~~~\mbox{on}~\partial\Omega\times(0,\bar{T}], \label{eq:initialcondition1.2}\\
u(x,0)&=\varphi(x),~~~~~~~~~~~~\mbox{in}~\Omega,\label{eq:initialcondition1.3}
\end{align}
where the functions $f(x,t)$, $g(x,t)$ and $\varphi(x)$ are known in some specific spaces for well-posedness. The coefficient matrix $A(x)$ is symmetric positive, i.e., there exist two positive constants $\alpha_1$ and $\beta_1$ such that for each $w, v\in\mathbb{R}^d$
\begin{equation}\label{contandcoef}
\begin{aligned}
(Aw, v)&\leq \beta_1\Vert w\Vert\Vert v\Vert,\\
(Av, v)&\geq \alpha_1\Vert v\Vert^2.
\end{aligned}
\end{equation}

With different approximation spaces for weak gradient operator, WG with $(P_k(T), P_{k}(e), RT_k(T)),~k\geq0$ element and element $(P_{k+1}(T),P_{k+1}(e),[P_{k}(T)]^d)$ were developed for the parabolic equations in \cite{Li2013Weak} and \cite{Gao2014ON}, respectively,  first-order backward Euler full-discrete scheme being investigated. However, there are few publications on second-order fully discrete WG schemes. Based on $(P_k(T), P_{k}(e), RT_k(T))$ element \cite{Raviart1977A} and $\theta$-schemes, optimal convergence of the fully discrete OPWG approximations will be analyzed in this paper.
Note that we concern about double-valued weak functions and if the jumps go to zero along the interior edges, the usual WG method can be recovered.

The paper is organized as follows. In Sec. $2$, the semi-discrete and full-discrete OPWG schemes are introduced and the latter is unconditionally stable. In Sec. $3$, optimal convergence analysis is presented including error estimates in the $L^2$ and energy norms. Finally, numerical results demonstrate the efficiency and feasibility of the new method.

Throughout this paper, we denote by $\varepsilon$ an arbitrarily small positive constant, $\Vert\cdot\Vert$ the $L^2$-norm and $L^p(0,\bar{T};V)$ with $p\geq1$ the spaces with respect to time where $V$ represents Sobolev space (see details in \cite{Adams2003Sobolev} or \cite{Quarteroni1994}). Moreover, we use $C$ for a positive constant independent of mesh size $h$ and time step $\tau$.

\section{OPWG schemes and stability}
\label{sec2} Let $\mathcal{T}_h$ be a partition of domain $\Omega$ satisfying shape regularity in \cite{Wang2014A}. For each element $T\in\mathcal{T}_h$, $h_T$ is its diameter and $h=\max_{T\in\mathcal{T}_h}h_T$ is the mesh size of $\mathcal{T}_h$. Denote by $\mathcal{E}_{I}$ the set of interior edges or flat faces, and $\partial T$ the edges or flat faces of element $T$. Let ${P}_{k}(T)$ be the space of polynomials of degree less than or equal to $k$ in variables. The weak Galerkin finite element space for OPWG is defined as
\begin{equation*}
\begin{aligned}
 V_{h}:=&\{(v_{0},v_{b}):v_{0}|_{T}\in{P}_{k}(T),T\in \mathcal{T}_{h};v_{b}|_{e}\in{P}_{k}(e)\times {P}_{k}(e),e\in \mathcal{E}_{I};\\
 &v_{b}|_{e}\in{P}_{k}(e),e\in\partial\Omega,k\geq0\},
 \end{aligned}
\end{equation*}
in particular, $V_h^0=\{v\in V_h~\mbox{and}~v_b=0~\mbox{on}~\partial\Omega\}$. For each $v=\{v_0,v_b\}\in V_h$, we define a unique local weak gradient $\nabla_wv\in RT_k(T)$ on each element $T\in\mathcal{T}_h$ satisfying
\begin{equation*}
(\nabla_wv,q)_T=-(v_0,\nabla \cdot q )_T+\langle v_b,q\cdot \mathbf{n}\rangle_{\partial T}, ~\forall~q\in RT_k(T).
\end{equation*}
In addition, we define several local $L^2$ projection operators onto space $V_h$. For each $T\in\mathcal{T}_h$ and $e\in\partial T$, let $Q_0$, $Q_b$ be the $L^2$ projection operators to $P_k(T)$ and $P_{k}(e)$, respectively. Denote $Q_hv$ by $Q_hv=\{Q_0v,Q_bv\},~\forall v\in L^2(T)$. Meantime, define $R_h$ the $L^2$ projection onto $RT_k(T)$, and then one can obtain
$$\nabla_w(Q_hv)=R_h(\nabla v),~\forall v\in H^1(T).$$
 Furthermore, we define a div projection $\Pi_h$ for $\mathbf{q}\in H(div;\Omega)$ satisfying that $\Pi_h \mathbf{q}\in H(div;\Omega)$ and $\Pi_h \mathbf{q}\in RT_k(T)$ on each element $T$, and (see \cite{Wang2013A} and \cite{Brezzi1991Mixed})
\begin{equation}\label{eq:piprojection}
(\nabla\cdot \mathbf{q},v_0)_T=(\nabla\cdot\Pi_h \mathbf{q},v_0)_T,~\forall~v_0\in P_k(T).
\end{equation}
Then, an approximation property of the projection $\Pi_h $ is given.
\begin{lemma}\label{lem:piprojection}\cite{Wang2013A}
For $u\in H^{k+2}(\Omega),~k\geq0$, it holds \begin{equation}\label{eq:piproestimate}
\Vert\Pi_h(A\nabla u)-A\nabla_w(Q_hu)\Vert\leq Ch^{k+1}\Vert u\Vert_{k+2}.
\end{equation}
\end{lemma}

For the sake of achieving the scheme of OPWG for parabolic problem \eqref{eq:problem1.1}, it is necessary to define a weak bilinear form in the following equation $$a_w(v,\chi):=(A\nabla_wv,\nabla_w\chi)+J_0(v,\chi),~\forall~v, \chi\in V_h,$$
where the penalty term is well defined as
$$J_{0}(v,\chi):=\sum\limits_{e\in \mathcal{E}_{I}}| e|^{-\beta_{0}}\langle\llbracket v_b\rrbracket,\llbracket \chi_b\rrbracket\rangle_{e},~\beta_0\geq 1.$$
 Let $e\in \mathcal{E}_{I}$ be shared by adjacent elements $T_1$ and $T_2$, then we define the jump on $e$ by $\llbracket v_b\rrbracket=v_b|_{T_1\cap e}-v_b|_{T_2\cap e}$.

The semi-discrete OPWG scheme for \eqref{eq:problem1.1}-\eqref{eq:initialcondition1.3} is to seek $u_h(t)\in V_h$ satisfying the boundary condition $u_h(x, t)=Q_hg(x, t)$ on $\partial\Omega\times(0,\bar{T}]$ and the initial condition $u_h(0)=Q_h\varphi$ such that
\begin{equation}\label{eq:semi-disc-format}
\begin{aligned}
((u_0)_t,v_0)+a_w(u_h, v_h)=(f, v_0),~~\forall~v_h\in V_h^0.
\end{aligned}
\end{equation}
Now, we define energy norm as for any $v\in V_h$
$$\interleave v\interleave^2:=a_w(v, v).$$
The existence and uniqueness of semi-discrete solution of \eqref{eq:semi-disc-format} are obtained from coercivity and continuity of $a_w$ (see Lemma $3.1$ in \cite{liu2018an}).

Next, we present full-discrete OPWG schemes. The interval $(0,\bar{T}]$ is divided into subintervals by time step $\tau$ uniformly, i.e. $t^n=n*\tau$. With the $\theta$-schemes applied, the full-discrete OPWG schemes are to seek $u^n\in V_h$ satisfying the boundary condition $u^n=Q_hg(x, t^n)$ on $\partial\Omega\times(0,\bar{T}]$ and the initial condition $u^0=Q_h\varphi$ such that
\begin{equation}\label{eq:full-disc-format}
\begin{aligned}
&(\bar{\partial} u^n,v_0)+a_w(\theta u^{n}+(1-\theta)u^{n-1},v_h)\\
&=(\theta f(t^n)+(1-\theta)f(t^{n-1}),v_0),~~\forall~v_h\in V_h^0,
\end{aligned}
\end{equation}
where the parameters  $\theta$ vary in $[\frac{1}{2},1]$ and difference quotient $\bar{\partial} u^n:=\frac{u^n-u^{n-1}}{\tau}$. For simplification, we denote $f(t^n):=f(x,t^n)$. Moreover, when $\theta=1$, \eqref{eq:full-disc-format} is backward Euler scheme, and Crank-Nicolson (CN) scheme is recovered if $\theta=\frac{1}{2}$.

Let $K\in\Omega$ be a small subdomain. The flux in time interval $(t-\nabla t,t+\nabla t)$ holds
\begin{equation*}
\begin{aligned}
\int_{t-\nabla t}^{t+\nabla t}\int_Ku_tdxdt+\int_{t-\nabla t}^{t+\nabla t}\int_{\partial K}\mathbf{q}\cdot \mathbf{n} dsdt=\int_{t-\nabla t}^{t+\nabla t}\int_Kfdxdt,
\end{aligned}
\end{equation*}
where $\mathbf{q}=-A\nabla u$ is the flow rate of heat energy. Multiplying a test function $v=\{v_0,v_b=0\}$ such that $v_0=1$ in $K$ and $v_0=0$ elsewhere in \eqref{eq:semi-disc-format}, we can obtain that
\begin{equation*}
\begin{aligned}
\int_{t-\nabla t}^{t+\nabla t}\int_K(u_h)_tdxdt&-\int_{t-\nabla t}^{t+\nabla t}\int_{\partial K}R_h(A\nabla_wu_h)\cdot\mathbf{n}dsdt=\int_{t-\nabla t}^{t+\nabla t}\int_Kfdxdt.
\end{aligned}
\end{equation*}
which implies mass conservation, by taking a numerical flux $\mathbf{q}_h\cdot\mathbf{n}=-R_h(A\nabla_wu_h)\cdot\mathbf{n}$.
\subsection{Stability of the full-discrete scheme}
\label{sec2.1}At first, we will give the following Poincar$\acute{\mbox{e}}$-type inequality between the $L^2$ norm and  the energy norm.
\begin{lemma}\label{lem:poincare}For any $v\in V_h^0$, it holds
\begin{equation}\label{eq:equival1}
\begin{aligned}
\Vert v_0\Vert\leq C\interleave v\interleave.
\end{aligned}
\end{equation}
\end{lemma}
\begin{proof}Based on Theorem $2.1$ in \cite{Chen1998}, we know that the weak solution $\Psi\in H^1_0(\Omega)$ of elliptic problem $ \Delta\Psi=v_0$ satisfies $H^2$-regularity, i.e.
$\Vert \Psi\Vert_{2}\leq C\Vert v_0\Vert$. Denote by $\mathbf{q}=\nabla\Psi\in H(div;\Omega)$ and it is obvious that $\nabla\cdot\mathbf{q}=v_0$. With the use  of the definitions of $\Pi_h$ \eqref{eq:piprojection} and discrete weak gradient, trace inequality \cite{Wang2014A} and Cauchy-Schwarz inequality, then we can deduce that
\begin{equation}\label{eq:fistinequality}
\begin{aligned}
\Vert v_0\Vert^2&=\sum_{T\in\mathcal{T}_h}(v_0,\nabla\cdot\mathbf{q})_T
=\sum_{T\in\mathcal{T}_h}(v_0,\nabla\cdot\Pi_h\mathbf{q})_T\\
&=-\sum_{T\in\mathcal{T}_h}(\nabla_wv,\Pi_h\mathbf{q})_T+\sum_{e\in\mathcal{E}_{I}}\langle \llbracket v_b\rrbracket,\Pi_h\mathbf{q}\cdot\mathbf{n_e}\rangle_{e}\\
&\leq\Vert \nabla_wv\Vert\Vert\Pi_h\mathbf{q}\Vert+\sum_{e\in\mathcal{E}_{I}}\Vert \llbracket v_b\rrbracket\Vert_e\Vert\Pi_h\mathbf{q}\Vert_e\\
&\leq\Vert \nabla_wv\Vert\Vert\Pi_h\mathbf{q}\Vert+\Big(\sum_{e\in\mathcal{E}_{I}}\Vert \llbracket v_b\rrbracket\Vert_e\Big)h^{-\frac{1}{2}}\Vert\Pi_h\mathbf{q}\Vert,
\end{aligned}
\end{equation}
where $\mathbf{n}_e$ is a unit normal on $e$. From Lemma \ref{lem:piprojection}, it follows
\begin{equation}\label{eq:secondinequality}
\begin{aligned}
\Vert\Pi_h\mathbf{q}\Vert&=\Vert\Pi_h\nabla\Psi\Vert\leq\Vert\Pi_h\nabla\Psi-\nabla_w(Q_h\Psi)\Vert+\Vert\nabla_w(Q_h\Psi)\Vert\\
&\leq Ch\Vert\Psi\Vert_2+\Vert R_h\nabla\Psi\Vert\\
&\leq C\Vert v_0\Vert.
\end{aligned}
\end{equation}
Combining \eqref{eq:fistinequality} with \eqref{eq:secondinequality} leads to
\begin{equation*}
\begin{aligned}
\Vert v_0\Vert&\leq C\Big(\Vert \nabla_wv\Vert+\sum_{e\in\mathcal{E}_{I}}h^{-\frac{1}{2}}\Vert \llbracket v_b\rrbracket\Vert_e\Big)\\
&\leq C\Big(\Vert \nabla_wv\Vert+h^{\frac{\beta_0(d-1)-1}{2}}\sum_{e\in\mathcal{E}_{I}}|e|^{-\frac{\beta_0}{2}}\Vert \llbracket v_b\rrbracket\Vert_e\Big)\\
&\leq C\interleave v\interleave.
\end{aligned}
\end{equation*}
\end{proof}
\begin{theorem}\label{th:stability}
Let $u^n$ be the numerical solution of \eqref{eq:full-disc-format}. Assume $g=0$ i.e. the parabolic problem is homogeneous problem and $\Vert f(t)\Vert$ is bounded in $[0,\bar{T}]$. Then there exists a positive constant such that
\begin{equation*}
\begin{aligned}
\Vert u^{n}\Vert\leq \Vert u^{0}\Vert+C\sup_{t\in[0,\bar{T}]}\Vert f(t)\Vert.
\end{aligned}
\end{equation*}
\end{theorem}
\begin{proof}
By Cauchy-Schwarz inequality and \eqref{eq:equival1}, taking
$$v_h=\theta u^{n}+(1-\theta)u^{n-1}=(\theta-\frac{1}{2})(u^{n}-u^{n-1})+\frac{1}{2}(u^{n}+u^{n-1})$$
in \eqref{eq:full-disc-format} yields
\begin{equation}\label{eq:full-dis-sta}
\begin{aligned}
&\frac{1}{2}\Vert u^{n}\Vert^2-\frac{1}{2}\Vert u^{n-1}\Vert^2+(\theta-\frac{1}{2})\Vert u^{n}-u^{n-1}\Vert^2+\tau\interleave\theta u^{n}+(1-\theta)u^{n-1} \interleave^2\\
&=\tau(\theta f(t^{n})+(1-\theta)f(t^{n-1}),\theta u^{n}+(1-\theta)u^{n-1})\\
&\leq \frac{\tau}{4\varepsilon}\Vert \theta f(t^{n})+(1-\theta)f(t^{n-1})\Vert^2+\varepsilon\tau
\interleave\theta u^{n}+(1-\theta)u^{n-1} \interleave^2,
\end{aligned}
\end{equation}
where $\varepsilon>0$. Let $\varepsilon=\frac{1}{4}$ in \eqref{eq:full-dis-sta}, then it follows
\begin{equation*}
\begin{aligned}
\Vert u^{n}\Vert^2\leq \Vert u^{n-1}\Vert^2+
 C\tau\Vert \theta f(t^{n})+(1-\theta)f(t^{n-1})\Vert^2.
\end{aligned}
\end{equation*}
Therefore, summing the above inequality from $1$ to $n$, and with the boundedness of source function $f$, we obtain that
\begin{equation*}
\begin{aligned}
\Vert u^{n}\Vert^2&\leq \Vert u^{0}\Vert^2+
 C\tau\sum_{j=1}^{n}\Vert \theta f(t^{j})+(1-\theta)f(t^{j-1})\Vert^2\\
 &\leq \Vert u^{0}\Vert^2+C\bar{T}\sup_{t\in [0,\bar{T}]}\Vert f(t)\Vert^2,
\end{aligned}
\end{equation*}
and then the conclusion follows.
\end{proof}
\section{Optimal convergence orders}\label{sec3}
Thanks  to elliptic projection, we will establish optimal convergence analysis of the fully discrete OPWG schemes.

For $v\in H^{k+2}(\Omega)$, we define an elliptic projection $E_hv\in V_h$ satisfying the following equation
\begin{equation}\label{eq:ellipticpro}
\begin{aligned}
a_w(E_hv,\chi)=-(\nabla\cdot A\nabla v,\chi_0), ~~\forall~\chi\in V_h^0,
\end{aligned}
\end{equation}
where $E_hv$ is the $L^2$ projection of the trace of $v$ on the boundary.
Then $E_hv$ is the OPWG approximation of the solution of the elliptic problem
\begin{equation*}
\begin{aligned}
-\nabla\cdot(A\nabla v)&=f^*,~~\mbox{in}~\Omega,\\
v&=g,~~~~\mbox{on}~\partial\Omega.
\end{aligned}
\end{equation*}
The following error estimates for the elliptic projection will be used later (see \cite{liu2018an}).
\begin{lemma}\label{lem:elconverg}\cite{liu2018an}
Let $u\in H^{k+2}(\Omega),~k\geq0$,
then there exists a positive constant $C$ such that
\begin{equation}\label{eq:ellconverge}
\begin{aligned}
\interleave Q_hu-E_hu\interleave&\leq C(h^{k+1}+h^{\frac{\beta_0(d-1)-1}{2}})\Vert u\Vert_{k+2},\\
\Vert Q_hu-E_hu\Vert&\leq C(h^{k+2}+h^{\frac{\beta_0(d-1)+1}{2}}+h^{\beta_0(d-1)-1})\Vert u\Vert_{k+2}.
\end{aligned}
\end{equation}
\end{lemma}
\subsection{Convergence of the semi-discrete scheme}
Denote the error of the semi-discrete scheme \eqref{eq:semi-disc-format} by $e_h:=Q_hu-u_h$. With the use of Lemma \ref{lem:elconverg}, error estimates can be derived as follows.
\begin{theorem}\label{the:semiconverg}
Let $u(t)$ and $u_h(t)\in V_h$ be the exact solution of \eqref{eq:problem1.1}-\eqref{eq:initialcondition1.3} and the numerical solution of \eqref{eq:semi-disc-format}, respectively. Assume $u\in L^1(0,\bar{T};H^{k+2}(\Omega))$, $u_t\in L^1(0,\bar{T};H^{k+2}(\Omega))$ and $\varphi\in H^{k+2}(\Omega)$ where $k\geq0$, then there exists a positive constant $C$ such that
\begin{equation}\label{eq:semi-l2estimate}
\begin{aligned}
\Vert e_h(t)\Vert\leq C(h^{k+2}+h^{\frac{\beta_0(d-1)+1}{2}}+h^{\beta_0(d-1)-1})(\Vert \varphi\Vert_{k+2}+\Vert u(t)\Vert_{k+2}+\int_0^t\Vert u_t\Vert_{k+2}ds),
\end{aligned}
\end{equation}
and
\begin{equation}\label{eq:semi-enerestimate}
\begin{aligned}
\interleave e_h(t)\interleave\leq& C(h^{k+1}+h^{\frac{\beta_0(d-1)-1}{2}})(\Vert \varphi\Vert_{k+2}+\Vert u(t)\Vert_{k+2}+\int_0^t\Vert u_t\Vert_{k+2}ds).
\end{aligned}
\end{equation}
Moreover, optimal convergence orders appear when the penalty parameter satisfies $\beta_0(d-1)\geq 2k+3$.
\end{theorem}
\begin{proof}It is necessary to decompose $e_h$ into two items
 $$e_h=(Q_hu-E_hu)+(E_hu-u_h):=\rho+\eta.$$
 From Lemma \ref{lem:elconverg}, we just need to estimate $\eta$. On account of the semi-discrete scheme \eqref{eq:semi-disc-format}, the definition of $E_h$ and that of $Q_h$, we get the following identity for each $v\in V_h^0$ (see \cite{Gao2014ON})
\begin{equation}\label{eq:semi-erroreq}
\begin{aligned}
(\eta_t,v_0)+a_w(\eta,v)&=(E_hu_t,v_0)-((u_0)_t,v_0)+a_w(E_hu,v)-a_w(u_h,v)\\
&=(E_hu_t,v_0)-(f,v_0)+a_w(E_hu,v)=-(\rho_t,v_0).
\end{aligned}
\end{equation}
Thus, choosing $v=\eta$ in the above identity yields
\begin{equation*}
\begin{aligned}
\frac{1}{2}\frac{d}{dt}(\eta,\eta)+a_w(\eta,\eta)=-(\rho_t,\eta).
\end{aligned}
\end{equation*}
With Cauchy-Schwarz inequality and \eqref{eq:equival1}, integrating the above equation over $(0,t)$ on the both sides shows that
\begin{equation*}
\begin{aligned}
\Vert \eta(t)\Vert^2\leq\Vert \eta(0)\Vert^2+C\int_0^t\Vert \rho_t\Vert^2ds.
\end{aligned}
\end{equation*}
Applying triangle inequality and Lemma \ref{lem:elconverg} results in the estimate \eqref{eq:semi-l2estimate}.

Moreover, we estimate $\interleave e_h(t)\interleave$. By taking $v=\eta_t$ in \eqref{eq:semi-erroreq}, one obtains
\begin{equation*}
\begin{aligned}
\Vert\eta_t\Vert^2+a_w(\eta,\eta_t)=-(\rho_t,\eta_t).
\end{aligned}
\end{equation*}
Notice that with Cauchy-Schwarz inequality, it is easy to get
$$\frac{1}{2}\frac{d}{dt}a_w(\eta,\eta)\leq \frac{1}{2}\Vert\rho_t\Vert^2, $$
and then integrating the inequality on $(0,t)$ leads to
\begin{equation*}
\begin{aligned}
\interleave\eta(t)\interleave^2\leq\interleave\eta(0)\interleave^2+\int_0^t\Vert\rho_t\Vert^2ds.
\end{aligned}
\end{equation*}
Consequently, with the use of Lemma \ref{lem:elconverg}, \eqref{eq:semi-enerestimate} is proved.
\end{proof}
\subsection{Convergence of the full-discrete scheme}
For each $t^n\in (0,\bar{T}]$, we denote the error term of full-discrete schemes by $$e^n:=Q_hu(t^n)-u^n=(Q_hu(t^n)-E_hu(t^n))+(E_hu(t^n)-u^n):=\rho^n+\eta^n. $$
Fully discrete error estimates are given in following theorem.
\begin{theorem}\label{the:fullconver}
Let $u$ and $u^n$ be the exact solution of \eqref{eq:problem1.1}-\eqref{eq:initialcondition1.3} and the numerical solution of \eqref{eq:full-disc-format}, respectively. Assume $u\in L^1(0,\bar{T};H^{k+2}(\Omega))$, $u_t\in L^1(0,\bar{T};H^{k+2}(\Omega))$ and $\varphi\in H^{k+2}(\Omega)$ with $k\geq0$. \\
When $\frac{1}{2}<\theta\leq1$, assume $u_{tt}\in L^1(0,\bar{T};L^2(\Omega))$, then there exists a positive constant $C$ such that
\begin{equation}\label{eq:full-esti1-l2}
\begin{aligned}
\Vert e^n\Vert\leq& C(h^{k+2}+h^{\frac{\beta_0(d-1)+1}{2}}+h^{\beta_0(d-1)-1})(\Vert \varphi\Vert_{k+2}+\Vert u(t^n)\Vert_{k+2}\\
&+\int_0^{t^n}\Vert u_t\Vert_{k+2} ds)+C\tau M1,
\end{aligned}
\end{equation}
and
\begin{equation}\label{eq:full-esti1-ener}
\begin{aligned}
\interleave e^n\interleave\leq& C(h^{k+1}+h^{\frac{\beta_0(d-1)-1}{2}})(\Vert \varphi\Vert_{k+2}+\Vert u(t^n)\Vert_{k+2}+\int_0^{t^n}\Vert u_t\Vert_{k+2}ds)+C\tau M1,
\end{aligned}
\end{equation}
where $M1=\int_0^{t^n}\Vert u_{tt}\Vert ds$.

When $\theta=\frac{1}{2}$, assume $u_{ttt}\in L^1(0,\bar{T};L^2(\Omega))$, then there exists a positive constant $C$ such that
\begin{equation}\label{eq:full-esti2-l2}
\begin{aligned}
\Vert e^n\Vert\leq& C(h^{k+2}+h^{\frac{\beta_0(d-1)+1}{2}}+h^{\beta_0(d-1)-1})(\Vert \varphi\Vert_{k+2}
+\Vert u(t^n)\Vert_{k+2}\\
&+\int_0^{t^n}\Vert u_t\Vert_{k+2} ds)+C\tau^2M2,
\end{aligned}
\end{equation}
and
\begin{equation}\label{eq:full-esti2-ener}
\begin{aligned}
\interleave e^n\interleave\leq& C(h^{k+1}+h^{\frac{\beta_0(d-1)-1}{2}})(\Vert \varphi\Vert_{k+2}+\Vert u(t^n)\Vert_{k+2}+\int_0^{t^n}\Vert u_t\Vert_{k+2}ds)+C\tau^2M2,
\end{aligned}
\end{equation}
where $M2=\int_0^{t^n}\Vert u_{ttt}\Vert ds$.\\
Here, optimal convergence orders appear when the penalty parameter satisfies $\beta_0(d-1)\geq 2k+3$.
\end{theorem}
\begin{proof} Based on Lemma \ref{lem:elconverg}, it is required to estimate $\eta^n$ in the energy norm. For each $v\in V_h^0$, combining scheme \eqref{eq:full-disc-format} with the definition \eqref{eq:ellipticpro}, we have the following error equation
\begin{equation}\label{eq:full_dis_eq}
\begin{aligned}
&(\bar{\partial}\eta^n,v_0)+a_w(\theta\eta^{n}+(1-\theta)\eta^{n-1},v)\\
&=(\bar{\partial}E_hu(t^n),v_0)-(\bar{\partial}u^n,v_0)
+a_w(\theta E_hu(t^{n})+(1-\theta)E_hu(t^{n-1}),v)\\
&~~~~-a_w(\theta u^{n}+(1-\theta) u^{n-1},v)\\
&=(\bar{\partial}E_hu(t^n),v_0)-(\nabla\cdot A(\nabla(\theta u(t^{n})+(1-\theta) u(t^{n-1}))),v_0)\\
&~~~~-(\theta f(t^{n})+(1-\theta)f(t^{n-1}),v_0)\\
&=(\bar{\partial}E_hu(t^n),v_0)-(\theta u_t(t^{n})+(1-\theta) u_t(t^{n-1}),v_0)\\
&=-(\bar{\partial}\rho^n,v_0)+(\bar{\partial}Q_hu(t^n)-(\theta u_t(t^{n})+(1-\theta) u_t(t^{n-1})),v_0).
\end{aligned}
\end{equation}
By integration by parts, we can deduce the following identities when $\frac{1}{2}<\theta\leq1$,
\begin{equation}\label{eq:thetaestimate}
\begin{aligned}
\bar{\partial}u(t^n)-(\theta u_t(t^{n})+(1-\theta) u_t(t^{n-1}))=-\frac{1}{\tau}\int_{t^{n-1}}^{t^{n}}(s-(1-\theta)t^{n}-\theta t^{n-1})u_{tt}ds,
\end{aligned}
\end{equation}
and when $\theta=\frac{1}{2}$,
\begin{equation}\label{eq:thetaestimate2}
\begin{aligned}
\bar{\partial}u(t^n)-\frac{1}{2}(u_t(t^{n})+ u_t(t^{n-1}))=\frac{1}{2\tau}\int_{t^{n-1}}^{t^{n}}(t^n-s)(t^{n-1}-s)u_{ttt}ds.
\end{aligned}
\end{equation}
(i) In the case $\frac{1}{2}<\theta\leq1$. Taking $v=\theta\eta^{n}+(1-\theta)\eta^{n-1}$ and substituting \eqref{eq:thetaestimate} into \eqref{eq:full_dis_eq}, it holds
\begin{equation*}
\begin{aligned}
&\frac{1}{2\tau}\Vert\eta^{n}\Vert^2-\frac{1}{2\tau}\Vert\eta^{n-1}\Vert^2
+\frac{1}{\tau}(\theta-\frac{1}{2})\Vert\eta^{n}-\eta^{n-1}\Vert^2+\interleave \theta\eta^{n}+(1-\theta)\eta^{n-1}\interleave^2\\
&\leq-(\bar{\partial}\rho^n,\theta\eta^{n}+(1-\theta)\eta^{n-1})
+\tau^{\frac{1}{2}}(\int_{t^{n-1}}^{t^{n}}\Vert u_{tt}\Vert^2 ds)^{\frac{1}{2}}\Vert\theta\eta^{n}+(1-\theta)\eta^{n-1}\Vert.
\end{aligned}
\end{equation*}
Furthermore, with the use of Cauchy-Schwarz inequality, the above inequality can be rewritten as
\begin{equation}\label{eq:full_dis_ineq}
\begin{aligned}
&\Vert\eta^{n}\Vert^2-\Vert\eta^{n-1}\Vert^2
+2(\theta-\frac{1}{2})\Vert\eta^{n}-\eta^{n-1}\Vert^2+2\tau\interleave \theta\eta^{n}+(1-\theta)\eta^{n-1}\interleave^2\\
&\leq2\tau\Vert\bar{\partial}\rho^n\Vert\Vert\theta\eta^{n}+(1-\theta)\eta^{n-1}\Vert
+2\tau^{\frac{3}{2}}(\int_{t^{n-1}}^{t^{n}}\Vert u_{tt}\Vert^2 ds)^{\frac{1}{2}}\Vert\theta\eta^{n}+(1-\theta)\eta^{n-1}\Vert\\
&\leq\frac{\tau}{\varepsilon}\Vert\bar{\partial}\rho^n\Vert^2
+2\varepsilon\tau\Vert\theta\eta^{n}+(1-\theta)\eta^{n-1}\Vert^2
+\frac{\tau^2}{\varepsilon}\int_{t^{n-1}}^{t^{n}}\Vert u_{tt}\Vert^2 ds.
\end{aligned}
\end{equation}
 Therefore, by \eqref{eq:equival1} and $\frac{1}{2}<\theta\leq1$, choosing $\varepsilon=\frac{1}{4}$ in \eqref{eq:full_dis_ineq} leads to
\begin{equation}\label{eq:thetabig2}
\begin{aligned}
\Vert\eta^{n}\Vert^2-\Vert\eta^{n-1}\Vert^2
\leq C\tau\Vert\bar{\partial}\rho^n\Vert^2+C\tau^2\int_{t^{n-1}}^{t^{n}}\Vert u_{tt}\Vert^2 ds.
\end{aligned}
\end{equation}
By using $\Vert\bar{\partial}\rho^n\Vert^2\leq\frac{C}{\tau}\int_{t^{n-1}}^{t^{n}}\Vert (Q_h-E_h)u_t\Vert^2 ds$ and summing \eqref{eq:thetabig2} from $1$ to $n$, it holds
\begin{equation*}
\begin{aligned}
\Vert\eta^{n}\Vert^2\leq\Vert\eta^{0}\Vert^2
+C\int_0^{t^n}\Vert\rho_t\Vert^2+C\tau^2\int_0^{t^n}\Vert u_{tt}\Vert^2 ds.
\end{aligned}
\end{equation*}
The error estimate \eqref{eq:full-esti1-l2} in $L^2$-norm follows owing to $\Vert\eta^{0}\Vert=\Vert\rho^{0}\Vert$ and Lemma \ref{lem:elconverg}.

Concerning about the error estimate in energy norm, taking $v=\eta^{n}-\eta^{n-1}$ in \eqref{eq:full_dis_eq} leads to
\begin{equation*}
\begin{aligned}
&\frac{1}{\tau}\Vert\eta^{n}-\eta^{n-1}\Vert^2
+(\theta-\frac{1}{2})\interleave\eta^{n}-\eta^{n-1}\interleave^2
+\frac{1}{2}(\interleave\eta^{n}\interleave^2-\interleave\eta^{n-1}\interleave^2)\\
&=-(\bar{\partial}\rho^n,\eta^{n}-\eta^{n-1})+(\bar{\partial}Q_hu(t^n)-(\theta u_t(t^{n})+(1-\theta) u_t(t^{n-1})),\eta^{n}-\eta^{n-1})\\
&\leq C\tau\Vert\bar{\partial}\rho^n\Vert^2+C\tau^2\int_{t^{n-1}}^{t^{n}}\Vert u_{tt}\Vert^2 ds+\frac{\varepsilon}{\tau}\Vert\eta^{n}-\eta^{n-1}\Vert^2.
\end{aligned}
\end{equation*}
With $\theta>\frac{1}{2}$ and appropriate choice $\varepsilon=\frac{1}{2}$, we show that
\begin{equation}\label{eq:thetabig}
\begin{aligned}
\interleave\eta^{n}\interleave^2\leq\interleave\eta^{n-1}\interleave^2
+C\tau\Vert\bar{\partial}\rho^n\Vert^2+C\tau^2\int_{t^{n-1}}^{t^{n}}\Vert u_{tt}\Vert^2 ds.
\end{aligned}
\end{equation}
Finally, owing to Lemma \ref{lem:elconverg}, the result \eqref{eq:full-esti1-ener} follows immediately.

(ii) In the case $\theta=\frac{1}{2}$, by applying \eqref{eq:thetaestimate2} to the above process, the last term at the right hand side of \eqref{eq:thetabig2} and that of \eqref{eq:thetabig} become $C\tau^4\int_{t^{n-1}}^{t^n}\Vert u_{ttt}\Vert^2 ds$. Analogously, we can also prove the results \eqref{eq:full-esti2-l2} and \eqref{eq:full-esti2-ener}.
\end{proof}
\section{Numerical experiments}\label{sec:05}
In this section, we will give an example in 2D to verify our theory. Let $\Omega=(0,1)\times(0,1)$ and $\bar{T}=1$. We consider the backward Euler $(\theta=1)$ and CN $(\theta=\frac{1}{2})$ schemes on time discretization, respectively. The error estimates are established in time level $t^n=\bar{T}$. In the example, we set $A\in \mathbb{R}^2$ is an identity matrix, denote the convergence order by $O(h^\gamma+\tau^\sigma)$, and take the penalty parameter $\beta_0=2k+3$. The programming is implemented in Matlab while uniform triangular meshes are generated by Gmsh.
\begin{example}\label{ex:002}
The exact solution is $u=sin(2\pi (t^2+1)+\pi/2)sin(2\pi x+\pi/2)sin(2\pi y+\pi/2)$ as the same as in \cite{Li2013Weak} and the initial condition, the Dirichlet boundary condition and the source function $f$ are determined by exact solution.
\end{example}
We first consider the backward Euler OPWG scheme. Table \ref{TT01} and Table \ref{TT02} show the numerical convergence with respect to mesh sizes $h$ while the time steps are taken small enough. When taking $\tau=h^2$ in Table \ref{TT01} for $k=0$, the convergence orders are $O(h)$ and $O(h^2)$ in the energy norm and $L^2$-norm, respectively, which are in agreement with our analysis completely. Moreover, by choosing $\tau=h^3$ in Table \ref{TT02} for $k=1$, the convergence orders are $O(h^2)$ and $O(h^3)$ in the energy norm and $L^2$-norm, respectively. On the other hand, Table \ref{TT03} presents convergence orders about time step $\tau$ for $k=1$. When mesh size $h=1/128$ is fixed, the convergence orders on time are $O(\tau)$ in both energy norm and $L^2$-norm. Moreover, we illustrate the errors of Table \ref{TT03} in Fig. \ref{fig1} with $loglog$ functions. The least squares fitting method is used to get convergence rates.

\begin{table}[!th]
  \centering
  \caption{Convergence with respect to mesh size $h$ in backward Euler scheme}
  \label{TT01}
  \begin{tabular}{cccccccccc}
  \hline\noalign{\smallskip}
 \multicolumn{1}{c}{h} &\multicolumn{4}{c}{$\tau=h^2,~\theta=1,~k=0,~\beta_0=3$}\\
\cline{2-5}\noalign{\smallskip}
&$\interleave e_{h}\interleave$  &$\gamma$ &$\Vert e_{0}\Vert$&$\gamma$ \\
  \hline
 1/8  &2.4624e-01 &~~~~   &7.8378e-02&~~~~   \\
 1/16 &1.2250e-01&1.0072 &2.0368e-02&1.9441\\
 1/32 &6.1139e-02&1.0026 &5.1424e-03&1.9857\\
 1/64 &3.0554e-02&1.0007 &1.2888e-03&1.9964\\
 1/128&1.5275e-02&1.0001 &3.2239e-04&1.9991\\
  \hline
  \end{tabular}
\end{table}

\begin{table}[!th]
  \centering
  \caption{Convergence with respect to mesh size $h$ in backward Euler scheme}
  \label{TT02}
  \begin{tabular}{ccccccccc}
  \hline\noalign{\smallskip}
 \multicolumn{1}{c}{h} &\multicolumn{4}{c}{$\tau=h^3,~\theta=1,~k=1,~\beta_0=5$}\\
\cline{2-5}\noalign{\smallskip}
&$\interleave e_{h}\interleave$  &$\gamma$ &$\Vert e_{0}\Vert$&$\gamma$ \\
  \hline
 1/4  &1.5624e-01 &~~~~   &2.6917e-02 &~~~~   \\
 1/8  &3.9689e-02 &1.9769 &2.4406e-03 &3.4632\\
 1/16 &9.9530e-03 &1.9955 &2.6886e-04 &3.1823\\
 1/32 &2.4911e-03 &1.9983 &3.2656e-05 &3.0414\\
 1/64 &6.2309e-04 &1.9992 &4.0350e-06 &3.0167\\
  \hline
  \end{tabular}
\end{table}

\begin{table}[!th]
  \centering
  \caption{Convergence with respect to time step $\tau$ in backward Euler scheme}
  \label{TT03}
  \begin{tabular}{ccccccccc}
  \hline\noalign{\smallskip}
  \multicolumn{1}{c}{$\tau$} &\multicolumn{4}{c}{$h=1/128,~\theta=1,~k=1,~\beta_0=5$}\\
\cline{2-5}\noalign{\smallskip}
&$\interleave e_{h}\interleave$  &$\sigma$ &$\Vert e_{0}\Vert$&$\sigma$ \\
  \hline
$1/32  $&1.8566e-02&~~~~ &1.8572e-02 &~~~~   \\
$1/64  $&9.6512e-03&0.9438 &9.7364e-03 &0.9316 \\
$1/128 $&4.9120e-03&0.9743 &4.9825e-03 &0.9665 \\
$1/256 $&2.4872e-03&0.9817 &2.5281e-03 &0.9788 \\
$1/512 $&1.2522e-03&0.9900 &1.2667e-03 &0.9969 \\
  \hline
  \end{tabular}
\end{table}

\begin{figure}[!th]
\centering
\includegraphics[width = 13cm,height =8.0cm]{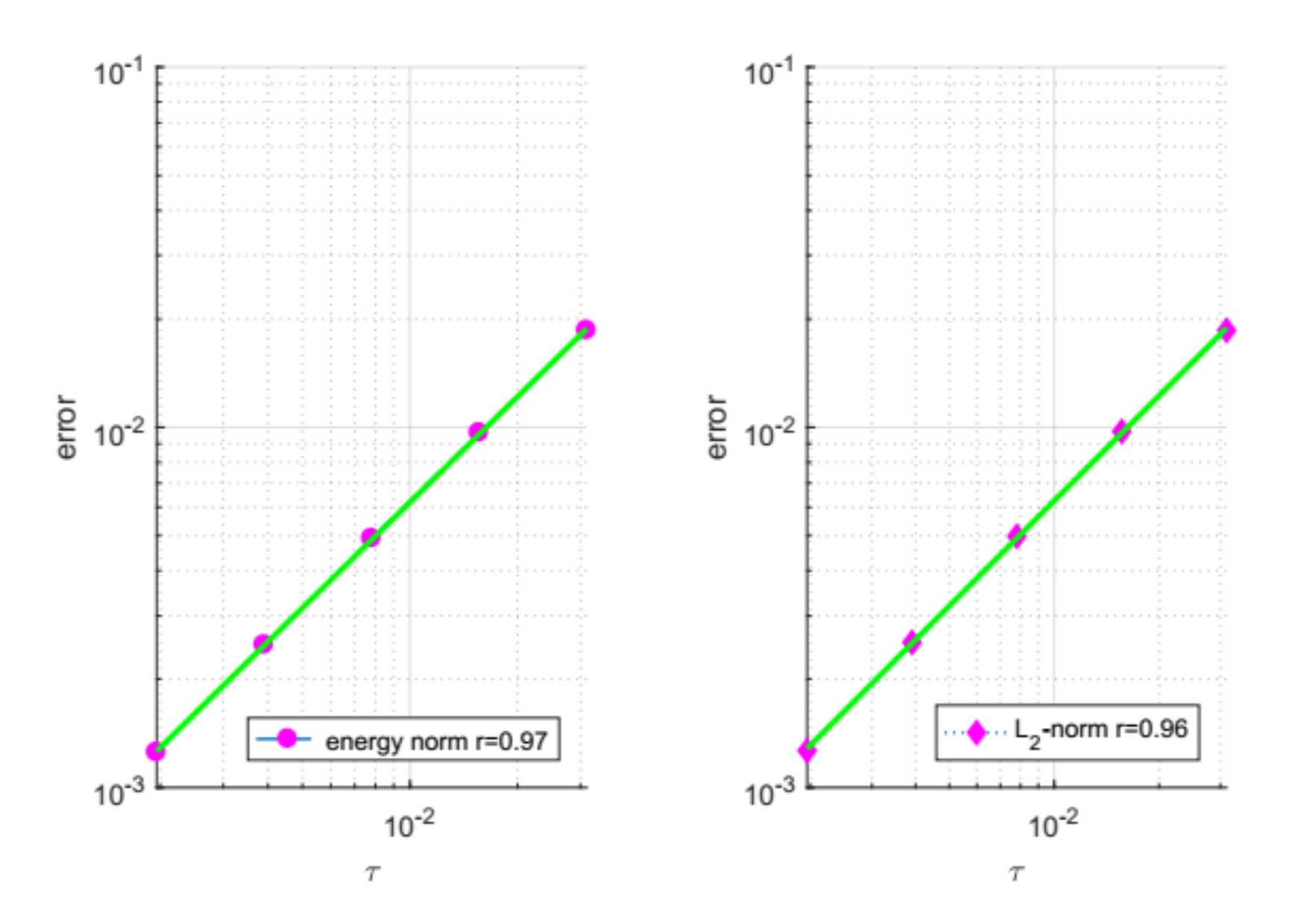}
\caption{Convergence with respect to time step $\tau$ in backward Euler scheme in Table \ref{TT03}}\label{fig1}
\end{figure}

Next, we apply the full-discrete CN scheme. Table \ref{TT04} shows that when $\tau=h^2$ and $k=0$, the convergence orders are $O(h^2)$ in the energy and $L^2$ norms  for CN scheme. It is interesting that  surperconvergence results in the energy norm are observed in Table \ref{TT04}. Moreover, with $\tau=h^3$ taken in Table \ref{TT05} for $k=1$, the convergence orders are $O(h^2)$ and $O(h^3)$ in the energy and $L^2$ norms, respectively. In Table \ref{TT06}, we consider the convergence order for CN scheme on time step $\tau$ while $k=1$. When the mesh size $h=1/128$ is fixed enough fine, the convergence orders are $O(\tau^2)$ in both energy and $L^2$ norms, which are in agreement with our theory. In Fig. \ref{fig2}, the errors in Table \ref{TT06} are plotted.

\begin{table}[!th]
  \centering
  \caption{Convergence with respect to mesh size $h$ of CN scheme}
  \label{TT04}
  \begin{tabular}{ccccccccc}
  \hline\noalign{\smallskip}
 \multicolumn{1}{c}{h} &\multicolumn{4}{c}{$\tau=h^2,~\theta=\frac{1}{2},~k=0,~\beta_0=3$}\\
\cline{2-5}\noalign{\smallskip}
&$\interleave e_{h}\interleave$  &$\gamma$ &$\Vert e_{0}\Vert$&$\gamma$\\
  \hline
 1/8  &8.1480e-02&~~~~   &7.8963e-02&~~~~  \\
 1/16 &2.1661e-02&1.9113 &2.0195e-02&1.9671\\
 1/32 &5.6280e-03&1.9444 &5.0780e-03&1.9916\\
 1/64 &1.4812e-03&1.9258 &1.2714e-03&1.9978\\
 1/128&4.0315e-04&1.8773 &3.1795e-04&1.9995\\
  \hline
  \end{tabular}
\end{table}

\begin{table}[!th]
  \centering
  \caption{Convergence with respect to mesh size $h$ of CN scheme}
  \label{TT05}
  \begin{tabular}{ccccccccc}
  \hline\noalign{\smallskip}
 \multicolumn{1}{c}{h} &\multicolumn{4}{c}{$\tau=h^3,~\theta=\frac{1}{2},~k=1,~\beta_0=5$}\\
\cline{2-5}\noalign{\smallskip}
&$\interleave e_{h}\interleave$  &$\gamma$ &$\Vert e_{0}\Vert$&$\gamma$\\
  \hline
 1/4  &3.5257e-02&~~~~   &2.7750e-02&~~~~  \\
 1/8  &8.5109e-03&2.0505 &2.2400e-03&3.6309\\
 1/16 &2.5646e-03&1.7305 &2.2664e-04&3.3050\\
 1/32 &6.7612e-04&1.9233 &2.6522e-05&3.0951\\
 1/64 &1.7167e-04&1.9776 &3.2647e-06&3.0221\\
  \hline
  \end{tabular}
\end{table}

\begin{table}[!th]
  \centering
  \caption{Convergence with respect to time step $\tau$ of CN scheme}
  \label{TT06}
  \begin{tabular}{ccccccccc}
  \hline\noalign{\smallskip}
  \multicolumn{1}{c}{$\tau$} &\multicolumn{4}{c}{$h=1/128,~\theta=\frac{1}{2},~k=1,~\beta_0=5$}\\
\cline{2-5}\noalign{\smallskip}
&$\interleave e_{h}\interleave$  &$\sigma$ &$\Vert e_{0}\Vert$&$\sigma$\\
  \hline
$1/4$  &1.2293e-01 &~~~~   &1.2361e-01 &~~~~    \\
$1/8$  &1.5203e-02 &3.0154 &1.9252e-02 &2.6827 \\
$1/16$ &2.5534e-03 &2.5738 &3.6999e-03 &2.3794\\
$1/32$ &5.8667e-04 &2.1217 &8.6930e-04 &2.0895 \\
$1/64$ &1.4530e-04 &2.0135 &2.1619e-04 &2.0075 \\
  \hline
  \end{tabular}
\end{table}

\begin{figure}[!th]
\centering
\includegraphics[width = 13cm,height =8.0cm]{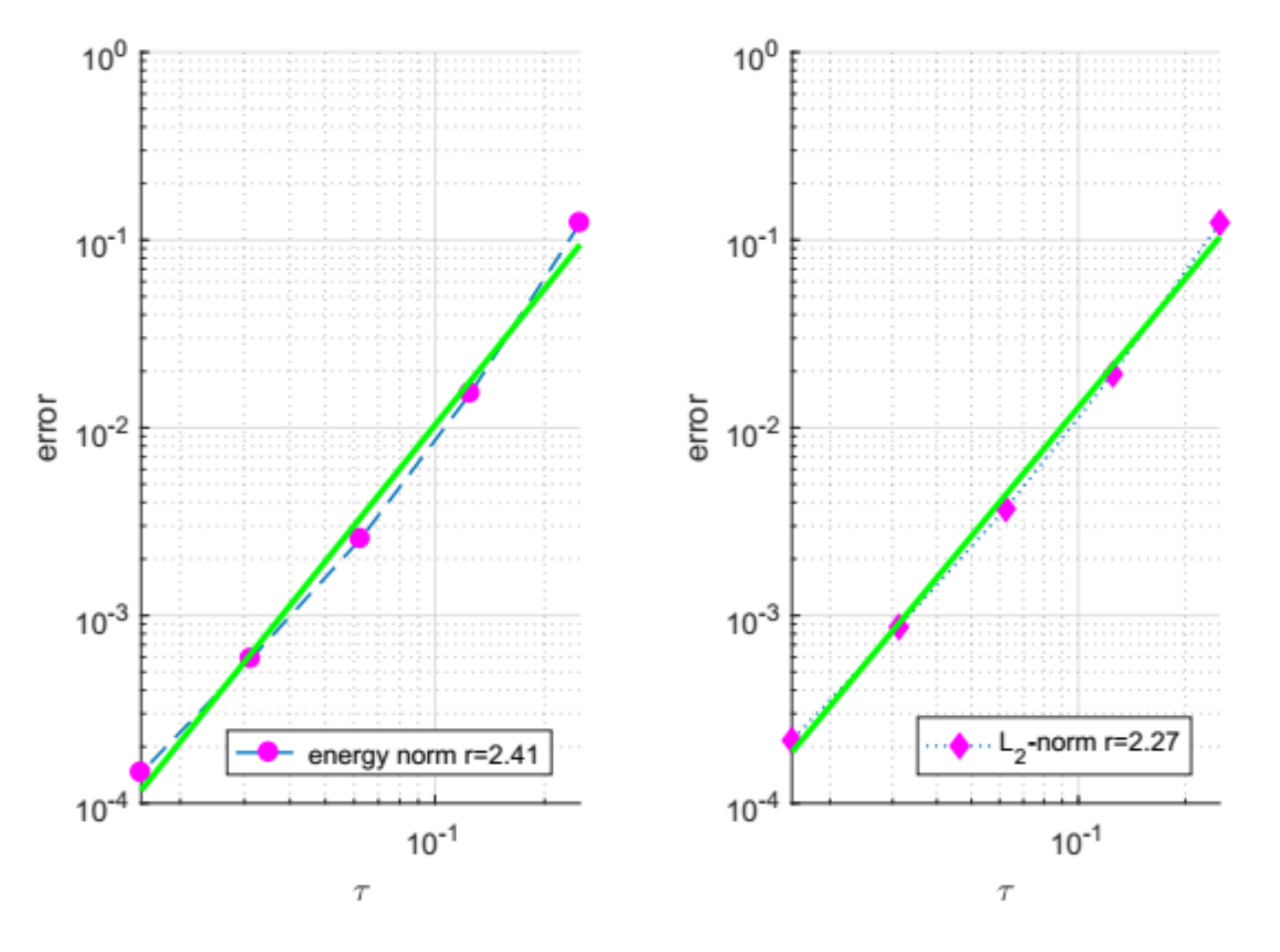}
\caption{Convergence order with respect to time step $\tau$ of CN scheme in Table \ref{TT06}}\label{fig2}
\end{figure}

\section*{Acknowledgements}
The author wishes to thank associate professor Lunji Song for his critical reading of the manuscript, helpful discussions and valuable suggestions. The research of the author is supported in part by the Natural Science Foundation of Gansu Province, China (Grant 18JR3RA290).

\end{document}